\documentclass[11pt,a4paper,reqno]{amsart} 

\usepackage [english]{babel}  

\DeclareFontFamily{U}{mathb}{\hyphenchar\font45}
\DeclareFontShape{U}{mathb}{m}{n}{
      <5> <6> <7> <8> <9> <10>
      <10.95> <12> <14.4> <17.28> <20.74> <24.88>
      mathb10
      }{}
\DeclareSymbolFont{mathb}{U}{mathb}{m}{n}

\DeclareMathSymbol{\sqbullet}{1}{mathb}{"0D}

\allowdisplaybreaks 

\usepackage{amsmath,mathtools}
\usepackage{amssymb}
\usepackage{euscript}
\usepackage{bm}
\usepackage{graphicx}
\usepackage[all]{xy}
\usepackage[dvipsnames]{xcolor}
\usepackage[colorlinks=true,linktocpage=true,pagebackref=false, citecolor=black,linkcolor=black]{hyperref}
\usepackage{pifont}
\usepackage{tikz,pgfplots}
\pgfplotsset{compat=1.10}
\usepgfplotslibrary{fillbetween}
\usetikzlibrary{cd,arrows,decorations.pathmorphing,backgrounds,automata,positioning,fit,matrix}
\usepackage{caption,subcaption}
\usepackage{comment}
\usepackage{xcolor}

\usetikzlibrary{matrix}
\usetikzlibrary{decorations.pathreplacing, calc, fit, backgrounds}

\pgfplotsset{soldot/.style={color=black,only marks,mark=*}} \pgfplotsset{holdot/.style={color=black,fill=white,only marks,mark=*}}

\newtheorem{thm}{Theorem}[section]

\newtheorem{prop}[thm]{Proposition}

\newtheorem{cor}[thm]{Corollary}

\newtheorem{quest2}[thm]{Question}

\theoremstyle{definition}
\newtheorem{defn}[thm]{Definition}

\theoremstyle{remark}

\newtheorem{remark}[thm]{Remark}

\newtheorem{example}[thm]{Example}

\numberwithin{equation}{section}
\numberwithin{figure}{section}

 \newcommand{\R}{{\mathbb R}}
 \newcommand{\C}{{\mathbb C}}

\newcommand{\x}{{\tt x}} \newcommand{\y}{{\tt y}} 
\newcommand{\z}{{\tt z}} \renewcommand{\t}{{\tt t}}
 
\newcommand{\w}{{\tt w}}

\usepackage{scalerel}

\begin{document}
\title[Holomorphic functions with Nash real part]{Holomorphic functions with Nash real part}

\begin{abstract}
In this paper, we show that a holomorphic function, defined on an open subset $D$ of $\C^n$, is a complex Nash function if and only if its real part (or equivalently its imaginary part) is a real Nash function. 
\end{abstract}
\subjclass[2020]{Primary: 14P20, 32C07. Secondary: 32A10}
\keywords{Real Nash functions, Complex Nash functions, Holomorphic functions}

\date{24/11/2025}
\author{Antonio Carbone}
\address{Dipartimento di Scienze dell'Ambiente e della Prevenzione, Palazzo Turchi di Bagno, C.so Ercole I D'Este, 32, Università di Ferrara, 44121 Ferrara (ITALY)}
\email{antonio.carbone@unife.it}

\maketitle

\section{Introduction}

Nash functions were introduced by John Nash in his groundbreaking paper \cite{n}. A real (resp. complex) Nash function is a real analytic (resp. holomorphic) function which is algebraic over the ring of polynomials with real (resp. complex) coefficients (see Definition \ref{defNash} below). Let $D\subset \C^n$ be an open subset and $f:=f_1+if_2:D\to \C$ a holomorphic function. After identifying $\C^n$ with $\R^{2n}$ in the usual way, the functions $f_1$ and $f_2$ can be regarded as real analytic functions. A natural question is the following:

\begin{quest2}\label{q1}
Is there any relationship between the following two properties?
\begin{itemize}
\item[{\rm(i)}] $f$ is a complex Nash function,
\item[{\rm(ii)}] $f_1$ and $f_2$ are real Nash functions.
\end{itemize}
\end{quest2} 

In Proposition \ref{realcomplex}, we show that properties (i) and (ii) are equivalent. This result is probably well known by experts. We include its proof in this paper, as we could not find a precise reference for it in the existing literature. It is worthwhile to note that property (ii) implies property (i) if and only if the involved function $f:=f_1+if_2$ is holomorphic (see Remark \ref{hol1} below). 

A more interesting question is the following:

\begin{quest2}\label{q2}
Is it true that if $f_1$ is a real Nash function, then $f$ is a complex Nash function?
\end{quest2} 

Nash functions are not closed under integration. Thus, to answer the previous question, one cannot simply use the Cauchy-Riemann equations. 

The purpose of this paper is to answer Question \ref{q2} in the positive. Namely, to show the following:

\begin{thm}\label{main}
Let $D\subset \C^n$ be an open subset and $f:=f_1+if_2:D\to \C$ a holomorphic function. If $f_1$ is a real Nash function on $D$, then $f$ is a complex Nash function on $D$. 
\end{thm}

As $if=-f_2+if_1$ is a holomorphic function on $D$, then the previous result holds true also if we consider the imaginary part $f_2$ of $f$ instead of its real part $f_1$. By Proposition \ref{realcomplex} and Theorem \ref{main}, we deduce straightforwardly the following:

\begin{cor}
Let $D\subset \C^n$ be an open subset and $f:=f_1+if_2:D\to \C$ a holomorphic function. Then the following are equivalent:
\begin{itemize}
\item[{\rm(i)}] $f$ is a complex Nash function.
\item[{\rm(ii)}] $f_1$ is a real Nash function.
\item[{\rm(iii)}] $f_2$ is a real Nash function.
\end{itemize}
\end{cor}

\section{Preliminaries}

In this section we recall some basic concepts about Nash functions and we answer Question \ref{q1}. We refer the reader to \cite[\S8]{bcr} and \cite{t} for a more detailed introduction to Nash functions.

\subsection{Real and complex Nash functions} Let $k$ be either the field of real numbers $\R$ or the field of complex numbers $\C:=\{x+i y : x,y\in\R\}$, where $i^2=-1$. Let $D\subset k^n$ be a (non-empty) open subset and $f:D\to k$ a $k$-valued function. 

\begin{defn}\label{defNash}
We say that $f$ is a \textit{$k$-Nash function at $x_0\in D$} if there exist an open neighbourhood $U$ of $x_0$ in $D$ and a non-zero polynomial $P\in k[\x,\t]:=k[\x_1,\ldots,\x_n,\t]$ such that 
\begin{itemize}
\item $f$ is $k$-analytic on $U$,
\item $P(x,f(x))=0$ for each $x\in U$. 
\end{itemize}
The function $f$ is said to be a \textit{$k$-Nash function on $D$} if it is a Nash function at every point of $D$. We denote by $\mathcal{N}_k(D)$ the set of all $k$-Nash functions on $D$ . \hfill$\sqbullet$
\end{defn}

We will also call $\R$-Nash functions \textit{real Nash functions} and $\C$-Nash functions \textit{complex Nash functions}. As being $k$-analytic is a local property, we have that if a function $f$ is $k$-Nash on $D$, then $f$ is also $k$-analytic on $D$. By \cite[Cor.8.1.6]{bcr}, it follows that the set $\mathcal{N}_{\R}(D)$ endowed with the pointwise addition and multiplication of (real) functions is a subring of the ring of real analytic functions on $D$. While, by \cite[Cor.1.11]{t}, the set $\mathcal{N}_{\C}(D)$ endowed with the pointwise addition and multiplication of (complex) functions is a subring of the ring of holomorphic functions on $D$. 

Let $f:D\to k$ be a $k$-analytic function and assume that $D$ is an open and connected subset of $k^n$. As for each polynomial $P\in k[\x,\t]$ the function $D\to k, \, x\mapsto P(x,f(x))$ is $k$-analytic, we deduce (straightforwardly) that the following are equivalent:
\begin{itemize}
\item[{\rm (i)}] $f\in\mathcal{N}_k(D)$,
\item[{\rm (ii)}] there exists $x_0\in D$ such that $f$ is $k$-Nash at $x_0$,
\item[{\rm (iii)}] there exists an irreducible polynomial $P\in k[\x,\t]$ such that $P(x,f(x))=0$ for each $x\in D$. Moreover, if $k=\C$, the polynomial $P$ is unique up to a non-zero constant (as a consequence of Hilbert's Nullstellensatz).
\end{itemize}
In particular, if $D$ has finitely many connected components, then $f\in\mathcal{N}_k(D)$ if and only if there exists a non-zero polynomial $P\in k[\x,\t]$ such that $P(x,f(x))=0$ for each $x\in D$. If $D$ has infinitely many connected components, $f$ might not be algebraic over the ring $k[\x,\t]$, as shown in the following example.

\begin{example}
For each integer $m\geq 1$ let 
$$
D_m:=\{x:=(x_1,\ldots,x_n)\in k^n : \|x-(4m,0,\ldots,0)\|_n<1\},
$$
where $\|\cdot\|_n$ denotes the Euclidean norm of $k^n$. The set $D:=\bigcup_{m\geq 1}D_m$ is an open subset of $k^n$ with infinitely many connected components. Let $f:D\to k$ be the $k$-analytic function defined as $f(x)=x_1^{m}$ for $x\in D_m$. Clearly $f\in\mathcal{N}_k(D)$, because for each integer $m\geq 1$ the polynomial $P_m(\x,\t):=\t-\x_1^m\in k[\x,\t]$ satisfies $P_m(x,f(x))=0$ for each $x\in D_m$. Assume now that there exists a non-zero polynomial $P\in k[\x,\t]$ such that $P(x,f(x))=0$ for each $x\in D$. Let $r\geq 0$ be the degree of $P$ with respect to the variable $\t$ and $s\geq 0$ the degree of $P$ with respect to the variable $\x_1$. We may write
$$
P(\x,\t)=p_r(\x)\t^r+p_{r-1}(\x)\t^{r-1}+\ldots+p_1(\x)\t+p_0(\x),
$$
where $p_0,\ldots,p_r\in k[\x]$ are polynomials of degree $\leq s$ with respect to the variable $\x_1$. Observe that $p_r$ is not the zero polynomial, because the degree of $P$ with respect to the variable $\t$ is $r$. As $P(x,x_1^m)=P(x,f(x))=0$ for each $x\in D_m$, then for each $m\geq 1$ the polynomial 
$$
P(\x,\x_1^m)=p_r(\x)\x_1^{mr}+p_{r-1}(\x)\x_1^{m(r-1)}+\ldots+p_1(\x)\x_1+p_0(\x)\in k[\x]
$$
is the zero polynomial, because $D_m$ is an open subset of $k^n$. Given a polynomial $Q\in k[\x]$ we denote by $\deg_{\x_1}(Q(\x))$ the degree of $Q(\x)$ with respect to the variable $\x_1$. Let $m\geq 1$ be an integer such that $m>s$. As $p_r$ is not the zero polynomial, then
$$
\deg_{\x_1}(p_r(\x)\x_1^{mr})\geq mr>m(r-1)+s\geq \deg_{\x_1}(P(\x,\x_1^m)-p_r(\x)\x_1^{mr})=\deg_{\x_1}(p_r(\x)\x_1^{mr}).
$$
In particular, it follows that $p_r(\x)=0$, which is a contradiction. We deduce that there exist no non-zero polynomials $P\in k[\x,\t]$ such that $P(x,f(x))=0$ for each $x\in D$. \hfill$\sqbullet$
\end{example}

Given a complex polynomial 
$
P(\z):=\sum_{k_1,\ldots,k_n=0}^s\alpha_{k_1,\ldots,k_n}\z_1^{k_1}\ldots\z_n^{k_n}\in \C[\z]:=\C[\z_1,\ldots,\z_n],
$
its \textit{conjugated polynomial} is defined as
$
P^c(\z)=\sum_{k_1,\ldots,k_n=0}^s\overline{\alpha_{k_1,\ldots,k_n}}\z_1^{k_1}\ldots\z_n^{k_n}.
$
A straightforward computation shows that the polynomial $P(\z)P^c(\z)$ has real coefficients. In particular, we have the following:

\begin{remark}\label{real}
Let $D\subset \C^n$ be an open and connected subset and $f:D\to \C$ a holomorphic function. Then $f\in\mathcal{N}_{\C}(D)$ if and only if there exists a non-zero polynomial $P\in\C[\z,\t]$ with real coefficients such that $P(z,f(z))=0$ for each $z\in D$. \hfill$\sqbullet$
\end{remark}

\subsection{Relationship between real and complex Nash functions}

Let $D\subset\C^n$ be an open subset and $f:D\to \C$ a holomorphic function. We may write $f=f_1+i f_2$, where $f_1,f_2:D\to \R$. Let $x:=(x_1,\ldots,x_n)$, $y:=(y_1,\ldots,y_n)$ and $z:=(z_1,\ldots,z_n)$. We identify $\C^n$ with $\R^{2n}$ by setting $z:=x+i y$. In particular, we may regard the subset $D\subset \C^n$ as the subset 
$$
\{(x,y)\in\R^{2n} :  x+i y\in D\}\subset \R^{2n}.
$$
Up to a light abuse of notation, we will still denote the previous set as $D$. Moreover, we may regard the holomorphic function $f:D\to \C$ as the real analytic map
$$
D\to \R^2, \quad (x,y)\mapsto (f_1(x+i y),f_2(x+i y)).
$$
Again, up to a light abuse of notation, we will still denote the previous analytic function as $f$ and its real components as $f_1(x,y)$ and $f_2(x,y)$. These choices will create no confusion as the situation will always be clear by the context.

The following result, that relates real and complex Nash functions, is probably well known. We include its proof here, as we could not find a precise reference for it.

\begin{prop}\label{realcomplex}
Let $D\subset\C^n$ be an open subset and $f:=f_1+i f_2:D\to\C$. The following are equivalent:
\begin{itemize}
\item[{\rm (i)}] $f\in\mathcal{N}_{\C}(D)$.
\item[{\rm (ii)}] $f$ is holomorphic on $D$ and $f_1,f_2\in\mathcal{N}_{\R}(D)$.
\end{itemize}
\end{prop}
\begin{proof}
As being Nash is a local property, we may assume that $D$ is connected. 

\noindent{\sc (i) implies (ii).} The functions $f_1,f_2:D\to \R$ are real analytic, because $f$ is holomorphic. In particular, it is enough to show that there exist two non-zero polynomials 
$$
P_1,P_2\in\R[\x,\y,\t]:=\R[\x_1,\ldots,\x_n,\y_1,\ldots,\y_n,\t]
$$ 
such that $P_1(x,y,f_1(x,y))=P_2(x,y,f_2(x,y))=0$ for each $(x,y)\in D$. As $f\in\mathcal{N}_\C(D)$ and $D$ is connected, then there exists a non-zero polynomial $P(\z,\t)\in\C[\z,\t]:=\C[\z_1,\ldots,\z_n,\t]$ such that $P(z,f(z))=0$ for each $z\in D$. By remark \ref{real}, we may assume that $P(\z,\t)$ has real coefficients. As $P(\z,\t)$ has real coefficients and $P(z,f(z))=0$ for each $z\in D$, then a straightforward computation shows that the polynomial $P(\overline{\z},\t)\in \C[\overline{\z},\t]:=\C[\overline{\z}_1,\ldots,\overline{\z}_n,\t]$ (anti-holomorphic in the variables $\overline{\z}_1,\ldots,\overline{\z}_n$) satisfies $P(\overline{z},\overline{f(z)})=0$ for each $z\in D$. Observe that the variables $\z_1,\ldots,\z_n$, $\overline{\z}_1,\ldots,\overline{\z}_n$ are algebraically independent over $\C$.

As $P(\z,\w)$ and $P(\overline{\z},\t-\w)$ have no common factors in $\C[\z,\overline{\z}, \t,\w]$ of positive degree with respect to the variable $\w$, then by \cite[3.6.1(ii)]{Cox} and its proof, there exist two polynomials $A,B\in\C[\z,\overline{\z}, \t,\w]$ and a non-zero polynomial $R\in\C[\z,\overline{\z},\t]$ such that
$$
R(\z,\overline{\z},\t)=A(\z,\overline{\z}, \t,\w)P(\z,\w)+B(\z,\overline{\z}, \t,\w)P(\overline{\z},\t-\w).
$$
Here the polynomial $R(\z,\overline{\z},\t)$ is the resultant in $\C[\z,\overline{\z}, \t,\w]$ with respect to the variable $\w$ of the polynomials $P(\z,\w)$ and $P(\overline{\z},\t-\w)$. Observe that
\begin{equation}\label{contoris}
R(z,\overline{z},f(z)+\overline{f(z)})=0
\end{equation}
for each $z\in D$, because $P(z,f(z))=P(\overline{z},\overline{f(z)})=0$ for each $z\in D$. 

For each $k=1,\ldots,n$ we write $\z_k:=\x_k+i\y_k$, so $\overline{\z}_k=\x_k-i\y_k$. As the ring homomorphism 
$$
\varphi:\C[\z,\overline{\z},\t]\mapsto \C[\x,\y,\t], \quad P(\z,\overline{\z},\t)\mapsto P(\x+i\y,\x-i\y,\t).
$$
is an isomorphism,  the polynomial 
$
Q(\x,\y,\t):=\varphi(R(\z,\overline{\z},\t))
$
is not the zero polynomial of $\C[\x,\y,\t]$. Up to substituting $Q$ with $QQ^c$, we may assume that $Q\in\R[\x,\y,\t]$. Moreover, as $2f_1(x,y)=f(z)+\overline{f(z)}$ for each $z=x+i y\in D$, we deduce that $Q(x,y,2f_1(x,y))=0$ for each $(x,y)\in D$. Up to substituting $Q(\x,\y,\t)$ with $Q(\x,\y,2\t)$, we may conclude that $f_1\in\mathcal{N}_\R(D)$. 

As $2i f_2(x,y)=f(z)-\overline{f(z)}=f(z)+(-\overline{f(z)})$ for each $z=x+i y\in D$, a similar argument shows that $f_2\in\mathcal{N}_\R(D)$, as required.

\noindent{\sc (ii) implies (i).} Let $P_1,P_2\in\R[\x,\y,\t]$ be non-zero polynomials such that
$$
P_1(x,y,f_1(x,y))=P_2(x,y,f_2(x,y))=0
$$
for each $(x,y)\in D$. The set $\{y_0\in \R^n: \, P_1(\x,y_0,\t)=  0 \, \,  \text{or} \, \, P_2(\x,y_0,\t)= 0\}$ is nowhere dense in $\R^n$. Thus, up to an affine change of coordinates, we may assume that $\Omega:=D\cap (\R^n\times\{0\})$ is a non-empty open subset of $\R^n\times \{0\}$ and that the polynomials 
$$
Q_1(\z,\t):=P_1(\z,0,\t) \quad   \text{and} \quad  Q_2(\z,\t):=P_2(\z,0,-i\t)
$$
are not the zero polynomials in $\C[\z,\t]$. By the fact that 
$$
P_1(x,0,f_1(x,0))=P_2(x,0,f_2(x,0))=0
$$
for each $(x,0)\in \Omega$, we deduce that 
$$
Q_1(x,f_1(x,0))=P_1(x,0,f_1(x,0))=0  \quad \text{and} \quad Q_2(x,i f_2(x,0))=P_2(x,0,f_2(x,0))=0
$$
for each $(x,0)\in \Omega$. 

Let $R(\z,\t)\in\C[\z,\t]$ be the resultant in $\C[\z,\t,\w]$ with respect to the variable $\w$ of $Q_1(\z,\w)$ and $Q_2(\z,\t-\w)$. As $Q_1(\z,\w)$ and $Q_2(\z,\t-\w)$ have no common irreducible factors in $\C[\z,\t,\w]$ of positive degree with respect to the variable $\w$, then by \cite[3.6.1(ii)]{Cox} and its proof, $R(\z,\t)$ is not the zero polynomial and there exist two polynomials $A,B\in\C[\z,\t,\w]$ such that
$$
R(\z,\t)=A(\z,\t,\w)Q_1(\z,\w)+B(\z,\t,\w)Q_2(\z,\t-\w).
$$
A straightforward computation shows that
\begin{equation}\label{Qbullet}
R(x,f_1(x,0)+i f_2(x,0))=0
\end{equation}
for each $(x,0)\in \Omega$. Recall that $f(z)=f_1(x,y)+i f_2(x,y)$ for each $z=x+i y\in D$. The map $D\to \C, \ z\mapsto R(z,f(z))$ is holomorphic on $D$, because $f$ is holomorphic on $D$. By \eqref{Qbullet}, we have
$$
R(z,f(z))=R(x,f_1(x,0)+i f_2(x,0))=0
$$
for each $z=x$ such that $(x,0)\in \Omega$. We deduce that $R(z,f(z))=0$ for each $z\in D$, because $\Omega$ is an open subset of the real vector space $\R^n\times \{0\}\subset \C^n$ and $R(z,f(z))=0$ on $\Omega$ (here we are using in an essential way that $f$ is holomorphic, see also Remark \ref{hol1} below). As $R(\z,\t)\in \C[\z,\t]$ is not the zero polynomial, we conclude that $f\in\mathcal{N}_\C(D)$, as required. 
\end{proof}

The following remark shows that the assumption that $f$ is holomorphic is essential in order to guarantee that (ii) implies (i). 

\begin{remark}\label{hol1}
Let $D\subset\C^n$ be an open subset, $D_0$ a connected component of $D$ and $f:D\to \C$ a continuous function. By \cite[Lem.3.B.13]{gr}, if there exists a non-zero polynomial $P\in\C[\z,\t]$ such that $P(z,f(z))=0$ for each $z\in D_0$, then $f$ is holomorphic on $D_0$. In particular, in the previous proposition, the assumption that $f$ is holomorphic cannot be dropped. In fact, if $f_1,f_2:D\to \R$ are real Nash functions on $D$ such that the map $f:=f_1+i f_2:D\to \C$ is not holomorphic on $D$, then there exists a connected component $D_0$ of $D$ such that there exist no non-zero polynomials $P\in\C[\z,\t]$ such that $P(z, f(z))=0$ for each $z\in D_0$. A concrete example is given by the anti-holomorphic function $f(z):=\overline{z}=x-iy$. \hfill$\sqbullet$
\end{remark}

\section{Proof of Theorem \ref{main}}

In this section, we prove Theorem \ref{main}. The argument is based on an elementary, but brilliant, calculation of Cartan \cite[\S IV.3.5]{Ca} that allows to reconstruct a holomorphic function by its real part in a definable way (that is, without integration). Recall that a set $U\subset\R^m$ is \em semialgebraic \em if it is a Boolean combination of sets defined by polynomial equalities and inequalities, while a function $f:U\to\R$ is \em semialgebraic \em if its graph is a semialgebraic subset of $\R^{m+1}$.

\begin{proof}[Proof of Theorem \ref{main}]
We consider the case $n=1$ and $n>1$ separately.

\noindent{\sc Case $n=1$.} We may assume that $0\in D$ and that $f(0)=0$. As being Nash is a local property, we may assume, in addition, that $D$ is an open ball centred in 0 such that the Taylor series of $f$ in 0 converges in $D$. Thus, we may write
$$
f(z)=\sum_{\nu=0}^\infty a_{\nu}z^\nu
$$
for each $z\in D$. Let $\rho>0$ be such that $D:=\{z\in \C : |z|<\rho\}$ and let $\Delta\subset \C^2$ be the polydisk defined as $\Delta:=\{(z_1,z_2)\in \C^2 : |z_1|<\tfrac{\rho}{2}, |z_2|<\tfrac{\rho}{2}\}$. Define the formal power series
$$
F(\z_1,\z_2):=\sum_{\nu=0}^\infty a_{\nu}(\z_1+i\z_2)^{\nu} \quad \text{and} \quad H(\z_1,\z_2):=\sum_{\nu=0}^\infty \overline{a}_{\nu}(\z_1-i\z_2)^{\nu}.
$$
A straightforward computation shows that these formal power series converge on $\Delta$. Define the holomorphic function $g:\Delta\to \C$ as
\begin{equation}\label{def}
g(z_1,z_2):=\frac{1}{2}(F(z_1,z_2)+H(z_1,z_2)).
\end{equation}
Observe that, as $f(0)=0$, then $H(0,0)=0$. By \eqref{def} and by the fact that $H(0,0)=0$, we have
$$
f(z)=2g\Big(\frac{z}{2},\frac{z}{2i}\Big)
$$
for each $z\in D$. As $f(x+iy)=F(x,y)$ and $\overline{f(x+iy)}=H(x,y)$ for each $(x,y)\in \Delta \cap \R^2$, by \eqref{def}, we have that $g(x,y)=f_1(x+iy)$ for each $(x,y)\in \Delta \cap \R^2$.

As $f_1\in \mathcal{N}_\R(D)$ and $D$ is connected, then there exists an irreducible non-zero polynomial $P\in \R[\x,\y,\t]$ such that 
$$
P(x,y,2f_1(x+iy))=0
$$ 
for each $(x,y)\in D$. In particular, $P(x,y,2g(x,y))=0$ for each $(x,y)\in \Delta\cap \R^2$, because $g(x,y)=f_1(x+iy)$ for each $(x,y)\in \Delta\cap \R^2$. As the function $(z_1,z_2)\mapsto P(z_1,z_2,2g(z_1,z_2))$ is holomorphic on $\Delta$ and vanishes identically on $\Delta\cap \R^2$, then $P(z_1,z_2,2g(z_1,z_2))=0$ for each $(z_1,z_2)\in \Delta$. 

We regard the polynomial $P$ as a polynomial in $\C[\z_1,\z_2,\t]$. If $P(\z,-i\z,\t)$ is the zero polynomial, then 
$
\{\z_2+i\z_1=0\}\subset \{P(\z_1,\z_2,\t)=0\}.
$
It follows, by Hilbert's Nullstellensatz, that $\z_2+i\z_1$ divides $(P(\z_1,\z_2,\t))^k$ in $\C[\z_1,\z_2,\t]$ for some integer $k\geq 1$. As $P$ is irreducible as a polynomial with real coefficients, then, up to a non-zero constant, $P(\z_1,\z_2,\t)=\z_1^2+\z_2^2$, which contradicts the fact that $P(x,y,2f_1(x+iy))=0$ for each $(x,y)\in D$. We deduce that $P(\z,-i\z,\t)$ is not the zero polynomial. In particular, the polynomial
\begin{equation}\label{eq1}
Q(\z,\t):=P\Big(\frac{\z}{2},\frac{\z}{2i},\t\Big)
\end{equation}
is not the zero polynomial in $\C[\z,\t]$.

Observe that if $(z_1,z_2)\in D\times D$, then $|z_1|<\rho$ and $|z_2|<\rho$, so $\big(\tfrac{z_1}{2},\tfrac{z_2}{2i})\in \Delta$, because $\Delta=\{(z_1,z_2)\in \C^2 : |z_1|<\tfrac{\rho}{2}, |z_2|<\tfrac{\rho}{2}\}$. As $P(z_1,z_2,2g(z_1,z_2))=0$ for each $(z_1,z_2)\in \Delta$, then
$$
P\Big(\frac{z_1}{2},\frac{z_2}{2i},2g\Big(\frac{z_1}{2},\frac{z_2}{2i}\Big)\Big)=0
$$
for each $(z_1,z_2)\in D\times D$. We deduce that
$$
Q(z,f(z))=Q\Big(z,2g\Big(\frac{z}{2},\frac{z}{2i}\Big)\Big)=P\Big(\frac{z}{2},\frac{z}{2i},2g\Big(\frac{z}{2},\frac{z}{2i}\Big)\Big)=0
$$
for each $z\in D$. As $Q$ is not the zero polynomial, we conclude that $f\in \mathcal{N}_\C(D)$, as required.

\noindent{\sc Case $n>1$.} As being Nash is a local property, we may assume that $D$ is an open ball of $\C^n$. We identify $\C^n$ with $\R^{2n}$ in the usual way. By Proposition \ref{realcomplex}, we only need to show that $f_2\in\mathcal{N}_\R(D)$. As $f_2$ is the imaginary part of the holomorphic function $f$, then $f_2$ is a real analytic function on $D$. In particular, as $D$ is a semialgebraic set, then by \cite[Prop.8.1.8]{bcr}, we may reduce to show: \textit{$f_2$ is a semialgebraic function on $D$.}

Let $v,w\in\C^n$ be such that $v\neq0$ and the real affine line $L_{v,w}:=\{w+xv : x\in\R\}$ satisfies $L_{v,w}\cap D\neq\varnothing$. The set $\pi_{v,w}:=\{w+xv+y(iv) : (x,y)\in\R^2\}$ is a complex affine plane that contains $L_{v,w}$. In particular, $D\cap\pi_{v,w}$ is a non-empty open subset of $\pi_{v,w}$. The restriction $f|_{D\cap \pi_{v,w}}$ is a holomorphic function on $D\cap\pi_{v,w}$, because $f$ is a holomorphic function on $D$. As $f_1\in \mathcal{N}_\R(D)$, then by \cite[Prop.8.1.8]{bcr}, $f_1$ is a semialgebraic function on $D$. Thus, the restriction $f_1|_{D\cap \pi_{v,w}}$ is a real analytic and semialgebraic function on the semialgebraic set $D\cap \pi_{v,w}$ (recall that $D$ is an open ball of $\C^n$). Using again \cite[Prop.8.1.8]{bcr}, we deduce that $f_1|_{D\cap \pi_{v,w}}\in \mathcal{N}_\R(D\cap \pi_{v,w})$. By the case $n=1$, we have that $f|_{D\cap \pi_{v,w}}\in\mathcal{N}_\C(D\cap \pi_{v,w})$. Thus, by Proposition \ref{realcomplex}, that $f_2|_{D\cap \pi_{v,w}}\in \mathcal{N}_\R(D\cap \pi_{v,w})$. In particular, arguing as before, we deduce that 
$$
f_2|_{D\cap L_{v,w}}=(f_2|_{D\cap \pi_{v,w}})|_{D\cap L_{v,w}}\in \mathcal{N}_\R(D\cap L_{v,w}).
$$ 
By \cite[Thm.1.1]{kks}, we conclude that $f_2$ is a semialgebraic function on $D$, as required.
\end{proof}

\subsection*{Acknowledgments} The author would like to thank Edward Bierstone for showing us the (brilliant) calculation of Cartan and Jos\'e F. Fernando for suggesting an argument that simplified our original proof of Proposition \ref{realcomplex}. The author is also strongly indebted to the anonymous referee for their careful reading and comments that have clarified an imprecise argument.

\bibliographystyle{amsalpha}

\end{document}